\documentclass[12pt]{amsart}
\usepackage{amssymb,latexsym,graphicx,amscd}
\usepackage{lscape}
\usepackage[all]{xy}
\usepackage{color}
\usepackage{epic,eepic}
\usepackage{xspace}
\usepackage{mathrsfs}
\usepackage{setspace}
\newtheorem{Thm}{Theorem}[section]
\newtheorem{Lem}[Thm]{Lemma}
\newtheorem{Cor}[Thm]{Corollary}
\newtheorem{Prop}[Thm]{Proposition}
\newtheorem{Conj}[Thm]{Conjecture}

\newtheorem{Problem}[Thm]{Problem}

\newcommand{\A}{\mathbb{A}}

\newcommand{\Z}{\mathbb{Z}}

\newcommand{\N}{\mathbb{N}}
\newcommand{\C}{\mathbb{C}}

\newcommand{\df}{\colon}

\newcommand{\cA}{{\mathcal A}}

\newcommand{\cC}{{\mathcal C}}

\newcommand{\cF}{{\mathcal F}}

\newcommand{\cK}{{\mathcal K}}
\newcommand{\cL}{{\mathscr L}}
\newcommand{\cM}{{\mathcal M}}

\newcommand{\cP}{{\mathcal P}}
\newcommand{\cR}{{\mathcal R}}
\newcommand{\cS}{{\mathcal S}}

\newcommand{\cX}{{\mathcal X}}

\newcommand{\g}{\mathfrak{g}}
\newcommand{\n}{\mathfrak{n}}
\newcommand{\h}{\mathfrak{h}}

\newcommand{\ba}{\mathbf{a}}

\newcommand{\bi}{{\mathbf i}}

\newcommand{\bx}{{\mathbf x}}
\newcommand{\by}{{\mathbf y}}
\newcommand{\bz}{{\mathbf z}}

\newcommand{\LL}{\Lambda}
\newcommand{\GG}{\Gamma}

\newcommand{\la}{\lambda}

\newcommand{\rk}{\operatorname{rank}}

\newcommand{\Ext}{\operatorname{Ext}}

\newcommand{\Quot}{\operatorname{Frac}}

\newcommand{\bsm}{\begin{smallmatrix}}
\newcommand{\esm}{\end{smallmatrix}}

\newcommand{\bbsm}{\left[\begin{smallmatrix}}
\newcommand{\besm}{\end{smallmatrix}\right]}

\newcommand{\bbm}{\begin{matrix}}
\newcommand{\ebm}{\end{matrix}}

\begin{document}


\title{Factorial cluster algebras}

\author{Christof Gei{\ss}}
\address{Christof Gei{\ss}\newline
Instituto de Matem\'aticas\newline
Universidad Nacional Aut{\'o}noma de M{\'e}xico\newline
Ciudad Universitaria\newline
04510 M{\'e}xico D.F.\newline
M{\'e}xico}
\email{christof@math.unam.mx}

\author{Bernard Leclerc}
\address{Bernard Leclerc\newline
Universit\'e ́ de Caen\newline
LMNO, CNRS UMR 6139\newline
Institut Universitaire de France\newline
14032 Caen cedex\newline
France}
\email{bernard.leclerc@unicaen.fr}

\author{Jan Schr\"oer}
\address{Jan Schr\"oer\newline
Mathematisches Institut\newline
Universit\"at Bonn\newline
Endenicher Allee 60\newline
53115 Bonn\newline
Germany}
\email{schroer@math.uni-bonn.de}



\begin{abstract}
We show that cluster algebras do not contain non-trivial
units and that all cluster variables are irreducible elements.
Both statements follow from Fomin and Zelevinsky's Laurent phenomenon.
As an application
we give a criterion for a cluster algebra to be
a factorial algebra.
This can be used to construct cluster
algebras, which are isomorphic to polynomial rings.
We also study various kinds of upper bounds
for cluster algebras, and we prove that
factorial cluster algebras coincide with their
upper bounds.
\end{abstract}

\maketitle

\setcounter{tocdepth}{1}

\tableofcontents

\parskip2mm



\section{Introduction and main results}


\subsection{Introduction}
The introduction of
cluster algebras by
Fomin and Zelevinsky \cite{FZ1} triggered an extensive theory.
Most results
deal with the combinatorics of seed and quiver mutations,
with
various categorifications of cluster algebras, and 
with cluster phenomena occuring in various areas
of mathematics, like representation theory of
finite-dimensional algebras, 
quantum groups and Lie theory, Calabi-Yau categories,
non-commutative Donaldson-Thomas invariants, Poisson geometry, discrete dynamical systems and
algebraic combinatorics.

On the other hand, there are not many results on cluster algebras themselves.
As a subalgebra of a field, any cluster algebra $\cA$ is
obviously an integral domain.
It is also easy to show that its field of fractions 
$\Quot(\cA)$ is 
isomorphic to a field $K(x_1,\ldots,x_m)$ of rational functions.
Several classes of cluster algebras are known to be finitely
generated, e.g. acyclic cluster algebras 
\cite[Corollary~1.21]{BFZ} and also a class of cluster
algebras arising from Lie theory \cite[Theorem~3.2]{GLSKM}.
Berenstein, Fomin and Zelevinsky gave an example
of a cluster algebra which is not
finitely generated. (One applies \cite[Theorem~1.24]{BFZ}
to the example mentioned in \cite[Proposition~1.26]{BFZ}.)
Only very little is known on further ring theoretic
properties of an arbitrary cluster algebra $\cA$.
Here are some basic questions we would like to address:
\begin{itemize}
\setlength{\itemsep}{9pt}

\item
Which elements in $\cA$ are invertible, irreducible or 
prime?

\item
When is $\cA$ a factorial ring?

\item
When is $\cA$ a polynomial ring?

\end{itemize}
In this paper, we work with cluster algebras of geometric type.

\subsection{Definition of a cluster algebra}
In this section we repeat Fomin and Zelevinsky's definition
of a cluster algebra.

A matrix $A = (a_{ij}) \in M_{n,n}(\Z)$ is 
\emph{skew-symmetrizable}
(resp. \emph{symmetrizable})
if there exists
a diagonal matrix 
$D = {\rm Diag}(d_1,\ldots,d_n) \in M_{n,n}(\Z)$ with positive diagonal
entries $d_1,\ldots,d_n$ 
such that $DA$ is skew-symmetric (resp. symmetric), i.e.
$d_ia_{ij} = -d_ja_{ji}$ (resp. $d_ia_{ij} = d_ja_{ji}$)
for all $i,j $.

Let $m$,$n$ and $p$ be integers with 
$$
m \ge p \ge n \ge 1
\text{\;\;\; and \;\;\;}
m > 1.
$$
Let
$B = (b_{ij}) \in M_{m,n}(\Z)$ be an $(m \times n)$-matrix
with integer entries.
By $B^\circ \in M_{n,n}(\Z)$ we denote the \emph{principal part} of $B$,
which is obtained from $B$ by deleting the last $m-n$ 
rows.

Let $\Delta(B)$ be the graph
with vertices $1,\ldots,m$ and an edge between $i$ and $j$
provided $b_{ij}$ or $b_{ji}$ is non-zero.
We call $B$ \emph{connected} if the graph $\Delta(B)$
is connected.

Throughout, we assume that $K$ is a field of
characteristic 0 or $K = \Z$.
Let $\cF := K(X_1,\ldots,X_m)$ be the field of rational
functions in $m$ variables.

A {\it seed} of $\cF$ is a pair $(\bx,B)$ such that
the following hold:
\begin{itemize}
\setlength{\itemsep}{9pt}

\item[(i)]
$B \in M_{m,n}(\Z)$,

\item[(ii)]
$B$ is connected,

\item[(iii)]
$B^\circ$ is skew-symmetrizable,
 
\item[(iv)]
$\bx = (x_1,\ldots,x_m)$ 
is an $m$-tuple of elements in $\cF$ such that
$x_1,\ldots,x_m$ are algebraically independent
over $K$.

\end{itemize}
For a seed $(\bx,B)$, the matrix $B$ is
the \emph{exchange matrix} of $(\bx,B)$.
We say that $B$ has \emph{maximal rank} if $\rk(B) = n$.

Given a seed $(\bx,B)$ and some $1 \le k \le n$  
we define the {\it mutation} of $(\bx,B)$
at $k$ as
$$
\mu_k(\bx,B) := (\bx',B'),
$$ 
where $B' = (b_{ij}')$ is defined as
$$
b_{ij}'
:=
\begin{cases}
- b_{ij} & \text{if $i=k$ or $j=k$},\\
b_{ij} + \dfrac{|b_{ik}|b_{kj} + b_{ik}|b_{kj}|}{2} &
\text{otherwise},
\end{cases}
$$
and $\bx' = (x_1',\ldots,x_m')$ is defined as
$$
x_s' :=
\begin{cases}
x_k^{-1}\prod_{b_{ik} > 0} x_i^{b_{ik}} + 
x_{k}^{-1}\prod_{b_{ik} < 0} x_i^{-b_{ik}} &
\text{if $s=k$},\\
x_s & \text{otherwise}.
\end{cases}
$$
The equality
\begin{equation}\label{mutationrelation}
x_kx_k' = \prod_{b_{ik} > 0} x_i^{b_{ik}} + 
\prod_{b_{ik} < 0} x_i^{-b_{ik}}
\end{equation}
is called an \emph{exchange relation}.
We write
$$
\mu_{(\bx,B)}(x_k) := x_k'
$$
and 
$$
\mu_k(B) := B'.
$$
It is easy to check that $(\bx',B')$ 
is again a seed.
Furthermore, we have $\mu_k\mu_k(\bx,B) = (\bx,B)$.

Two seeds $(\bx,B)$ and $(\by,C)$ are 
\emph{mutation equivalent}
if there exists a sequence $(i_1,\ldots,i_t)$ 
with $1 \le i_j \le n$ for all $j$ such that
$$
\mu_{i_t} \cdots \mu_{i_2}\mu_{i_1}(\bx,B) = (\by,C).
$$
In this case, we write
$(\by,C) \sim (\bx,B)$.
This yields an equivalence relation on all seeds of $\cF$.
(By definition $(\bx,B)$ is also mutation equivalent to itself.)

For a seed
$(\bx,B)$ of $\cF$ 
let 
$$
\cX_{(\bx,B)}
:= \bigcup_{(\by,C) \sim (\bx,B)} \{y_1,\ldots,y_n\},
$$ 
where the union is over all seeds $(\by,C)$ 
with
$(\by,C) \sim (\bx,B)$.
By definition, the \emph{cluster algebra} $\cA(\bx,B)$ 
associated to $(\bx,B)$ is the $L$-subalgebra of $\cF$ generated
by $\cX_{(\bx,B)}$, where
$$
L := K[x_{n+1}^{\pm 1},\ldots,x_p^{\pm 1},x_{p+1},\ldots,x_m]
$$
is the localization of the polynomial ring $K[x_{n+1},\ldots,x_m]$
at $x_{n+1} \cdots x_p$. (For $p=n$ we set
$x_{n+1} \cdots x_p := 1$.) 
Thus $\cA(\bx,B)$ is the $K$-subalgebra of
$\cF$ generated by 
$$
\{ x_{n+1}^{\pm 1},\ldots,x_p^{\pm 1},x_{p+1},\ldots,x_m \} \cup \cX_{(\bx,B)}.
$$
The elements of  $\cX_{(\bx,B)}$ are
the \emph{cluster variables} of $\cA(\bx,B)$.

We call $(\by,C)$ a \emph{seed} of $\cA(\bx,B)$
if $(\by,C) \sim (\bx,B)$.
In this case, for any $1 \le k \le n$ we call
$(y_k,\mu_{(\by,C)}(y_k))$ an \emph{exchange pair} of 
$\cA(\bx,B)$.
Furthermore, the $m$-tuple $\by$ is a \emph{cluster} of $\cA(\bx,B)$, and monomials of the
form $y_1^{a_1}y_2^{a_2} \cdots y_m^{a_m}$ with $a_i \ge 0$
for all $i$ are called
\emph{cluster monomials} of $\cA(\bx,B)$.

Note that for any cluster $\by$ of $\cA(\bx,B)$ we
have $y_i = x_i$ for all $n +1\le i \le m$.
These $m-n$ elements are the 
\emph{coefficients} of $\cA(\bx,B)$.
There are no invertible coefficients if $p=n$.

Clearly, for any two seeds of the form 
$(\bx,B)$ and $(\by,B)$ there
is an algebra isomorphism
$\eta\df \cA(\bx,B) \to \cA(\by,B)$ with
$\eta(x_i) = y_i$ for all $1 \le i \le m$, which respects the
exchange relations.
Furthermore, if $(\bx,B)$ and $(\by,C)$ are mutation equivalent seeds, then $\cA(\bx,B) = \cA(\by,C)$ and
we have $K(x_1,\ldots,x_m) = K(y_1,\ldots,y_m)$.

\subsection{Trivial cluster algebras and connectedness of 
exchange matrices}
Note that we always assume $m>1$.
For $m=1$ we would get the trivial cluster algebra
$\cA(\bx,B)$ with exactly two cluster variables, namely
$x_1$ and $x_1' := \mu_{(\bx,B)}(x_1) = x_1^{-1}(1+1)$.
In particular, for $K \not= \Z$, both cluster variables
are invertible in $\cA(\bx,B)$, and 
$\cA(\bx,B)$ is just the Laurent polynomial ring
$K[x_1^{\pm 1}]$.

Furthermore, for any seed $(\bx,B)$ of $\cF$
the exchange matrix $B$ is by definition connected. 
For non-connected $B$ one could write $\cA(\bx,B)$ as
a product $\cA(\bx_1,B_1) \times \cA(\bx_2,B_2)$ of
smaller cluster algebras and study the factors $\cA(\bx_i,B_i)$ separately.
The connectedness assumption 
also ensures that there are no 
exchange relations of the
form $x_kx_k' = 1+1$.

\subsection{The Laurent phenomenon}\label{introlaurent}
It follows by induction from the exchange relations
that for any cluster $\by$ of $\cA(\bx,B)$, any cluster 
variable $z$ of $\cA(\bx,B)$ is of the form
$$
z = \frac{f}{g},
$$
where $f,g \in \N[y_1,\ldots,y_m]$ are integer polynomials
in the cluster variables $y_1,\ldots,y_m$ with non-negative coefficients.
For any seed $(\bx,B)$ of $\cF$ let
$$
\cL_\bx := K[x_1^{\pm 1},\ldots,x_n^{\pm 1},
x_{n+1}^{\pm 1},\ldots,
x_p^{\pm 1},x_{p+1},\ldots,x_m]
$$
be the localization of  
$K[x_1,\ldots,x_m]$ at $x_1x_2\cdots x_p$,
and let
$$
\cL_{\bx,\Z} := \Z[x_1^{\pm 1},\ldots,x_n^{\pm 1},
x_{n+1},\ldots,x_m]
$$
be the localization of  
$\Z[x_1,\ldots,x_m]$ at
$x_1x_2 \cdots x_n$.
We consider $\cL_\bx$ and $\cL_{\bx,\Z}$ as subrings of the field $\cF$.
The following remarkable result, known as the \emph{Laurent phenomenon}, is due to Fomin and Zelevinsky and is our key tool to derive some ring theoretic
properties of cluster algebras.

\begin{Thm}[{{\cite[Theorem~3.1]{FZ1}},
{\cite[Proposition~11.2]{FZ2}}}]\label{Laurent}
For each seed $(\bx,B)$ of $\cF$ we
have 
$$
\cA(\bx,B) \subseteq \overline{\cA}(\bx,B) := 
\bigcap_{(\by,C) \sim (\bx,B)} \cL_{\by}
$$
and
$$
\cX_{(\bx,B)} \subset  
\bigcap_{(\by,C) \sim (\bx,B)} \cL_{\by,\Z}.
$$
\end{Thm}

The algebra $\overline{\cA}(\bx,B)$ is called the
\emph{upper cluster algebra} associated to $(\bx,B)$,
compare \cite[Section~1]{BFZ}.

\subsection{Upper bounds}
For a seed $(\bx,B)$ and $1 \le k \le n$ let
$(\bx_k,B_k) := \mu_k(\bx,B)$. 
Berenstein, Fomin and Zelevinsky \cite{BFZ} called
$$
U(\bx,B) := \cL_\bx \cap \bigcap_{k=1}^n \cL_{\bx_k}
$$
the \emph{upper bound} of $\cA(\bx,B)$.
They prove the following:

\begin{Thm}[{{\cite[Corollary~1.9]{BFZ}}}]
Let $(\bx,B)$ and $(\by,C)$ be mutation equivalent
seeds of $\cF$.
If $B$ has maximal rank and $p=m$,
then $U(\bx,B) = U(\by,C)$. 
In particular, we have
$\overline{\cA}(\bx,B) = U(\bx,B)$.
\end{Thm}

For clusters $\by$ and $\bz$ of $\cA(\bx,B)$
define 
$$
U(\by,\bz) := \cL_\by \cap \cL_\bz.
$$

\subsection{Acyclic cluster algebras}
Let $(\bx,B)$ be a seed of $\cF$ with
$B = (b_{ij})$.
Let
$\Sigma(B)$ be the quiver with
vertices $1,\ldots,n$, 
and 
arrows $i \to j$ for all $1 \le i,j \le n$
with $b_{ij} > 0$, compare \cite[Section~1.4]{BFZ}.
So $\Sigma(B)$ encodes the sign-pattern of the
principal part $B^\circ$ of $B$.

The seed $(\bx,B)$ and $B$ are called \emph{acyclic}
if $\Sigma(B)$ does not contain any oriented cycle.
The cluster algebra $\cA(\bx,B)$ is
\emph{acyclic} if there exists an acyclic seed $(\by,C)$
with $(\by,C) \sim (\bx,B)$.

\subsection{Skew-symmetric exchange matrices and quivers}
Let $B = (b_{ij})$ be a matrix in $M_{m,n}(\Z)$ 
such that $B^\circ$ is skew-symmetric.
Let $\Gamma(B)$ be the quiver with vertices $1,\ldots,m$
and $b_{ij}$ arrows $i \to j$ if 
$b_{ij}>0$, and $-b_{ij}$ arrows $j \to i$ if $b_{ij}<0$.
Thus given $\Gamma(B)$, we can recover $B$.
In the skew-symmetric case one often works with quivers
and their mutations instead of exchange matrices.

\subsection{Main results}
For a ring $R$ with $1$, let $R^\times$ be the set of invertible
elements in $R$.
Non-zero rings without zero divisors are called
\emph{integral domains}.
A non-invertible element $a$ in an integral domain $R$ is 
\emph{irreducible}
if it cannot be written as a product $a = bc$ with
$b,c \in R$ both non-invertible.
Cluster algebras are integral domains, since they
are by definition subrings of fields.

\begin{Thm}\label{Main1}
For any seed $(\bx,B)$ of $\cF$ the following hold:
\begin{itemize}
\setlength{\itemsep}{9pt}

\item[(i)]
We have
$\cA(\bx,B)^\times = 
\left\{ \la x_{n+1}^{a_{n+1}} \cdots x_p^{a_p} \mid 
\la \in K^\times, a_i \in \Z \right\}$.

\item[(ii)]
Any cluster variable in $\cA(\bx,B)$ is irreducible.

\end{itemize}
\end{Thm}

For elements $a,b$ in an integral domain $R$ we write
$a|b$ if there exists some $c \in R$ with $b = ac$.
A non-invertible element $a$ in a commutative ring $R$
is \emph{prime} if whenever $a|bc$ for some $b,c \in R$, then $a|b$ or $a|c$. Every prime element is irreducible, but the converse is not true in general.
Non-zero elements $a,b \in R$ are \emph{associate} if there
is some unit $c \in R^\times$ with $a = bc$.
An integral domain $R$ is \emph{factorial}
if the following hold:
\begin{itemize}
\setlength{\itemsep}{9pt}

\item[(i)]
Every non-zero non-invertible element $r \in R$ can be written as a product $r = a_1 \cdots a_s$ of irreducible elements $a_i \in R$.

\item[(ii)]
If $a_1 \cdots a_s = b_1 \cdots b_t$ with $a_i,b_j \in R$
irreducible for all $i$ and $j$, then $s=t$ and there is a bijection
$\pi\colon \{ 1,\ldots,s \} \to \{ 1,\ldots,s \}$ such that
$a_i$ and $b_{\pi(i)}$ are associate for all $1 \le i \le s$.

\end{itemize}
For example, any polynomial ring
is factorial.
In a factorial ring, all irreducible elements are prime.

Two clusters $\by$ and $\bz$ of a cluster algebra $\cA(\bx,B)$ are \emph{disjoint}
if $\{ y_1,\ldots,y_n \} \cap \{ z_1,\ldots,z_n \} = \varnothing$.

The next result gives a useful criterion when a cluster algebra
is a factorial ring.

\begin{Thm}\label{Main2}
Let $\by$ and $\bz$ be disjoint clusters of $\cA(\bx,B)$.
If there is a subalgebra $U$ of $\cA(\bx,B)$,
such that $U$ is
factorial and
$$
\{ y_1,\ldots,y_n,z_1,\ldots,z_n,x_{n+1}^{\pm 1},\ldots,
x_p^{\pm 1},x_{p+1},\ldots,x_m \} \subset U,
$$
then
$$
U = \cA(\bx,B) = U(\by,\bz).
$$
In particular, $\cA(\bx,B)$ is factorial and all cluster variables
are prime.
\end{Thm}

We obtain the following corollary on upper bounds of
factorial cluster algebras.

\begin{Cor}\label{Main3}
Assume that $\cA(\bx,B)$ is factorial.
\begin{itemize}
\setlength{\itemsep}{9pt}

\item[(i)]
If $\by$ and $\bz$ are
disjoint clusters of $\cA(\bx,B)$, then 
$\cA(\bx,B) = U(\by,\bz)$.

\item[(ii)]
For any $(\by,C) \sim (\bx,B)$ we have
$\cA(\bx,B) = U(\by,C)$.

\end{itemize}
\end{Cor}

In Section~\ref{Sec6.2}
we apply the above results to show that many
cluster algebras are polynomial rings.
In
Section~\ref{Sec7} we discuss
some further applications concerning the dual of Lusztig's
semicanonical basis and monoidal categorifications of
cluster algebras.

\subsection{Factoriality and maximal rank}
In Section~\ref{nonfac3} we give examples of 
cluster algebras $\cA(\bx,B)$, which are
not factorial. In these examples, $B$ does not have maximal rank.

After we presented our results at the Abel Symposium in Balestrand in June 2011, 
Zelevinsky asked the following question:

\begin{Problem}\label{question1}
Suppose $(\bx,B)$ is a seed of $\cF$ such that $B$ has maximal rank.
Does it follow that $\cA(\bx,B)$ is factorial?
\end{Problem}

After we circulated a first version of this article, Philipp Lampe
\cite{La}
discovered an example of a non-factorial cluster
algebra $\cA(\bx,B)$ with $B$ having maximal rank.
With his permission, 
we explain a generalization of his example in Section~\ref{nonfac2}.


\section{Invertible elements in cluster algebras}\label{Sec2}


In this section we prove Theorem~\ref{Main1}(i),
classifying the invertible elements of cluster algebras.

The following lemma is straightforward and well-known.

\begin{Lem}\label{lemma1}
For any seed $(\bx,B)$ of $\cF$ we have
$$
\cL_\bx^\times =
\{ \la x_1^{a_1}\cdots x_p^{a_p} \mid \la \in K^\times, a_i \in \Z \}.
$$
\end{Lem}

\begin{Thm}\label{Thminvert}
For any seed $(\bx,B)$ of $\cF$
we have
$$
\cA(\bx,B)^\times = 
\left\{ \la x_{n+1}^{a_{n+1}} \cdots x_p^{a_p} \mid 
\la \in K^\times, a_i \in \Z \right\}.
$$
\end{Thm}

\begin{proof}
Let $u$ be an invertible element in $\cA(\bx,B)$, and let
$(\by,C)$ be any seed of $\cA(\bx,B)$.
By the Laurent phenomenon Theorem~\ref{Laurent} we know that
$\cA(\bx,B) \subseteq \cL_\by$.
It follows that
$u$ is also invertible in $\cL_\by$.
Thus by Lemma~\ref{lemma1} there are $a_1,\ldots,a_p \in \Z$ and
$\la \in K^\times$ such that
$u = \la M$,
where 
$$
M = y_1^{a_1} \cdots y_k^{a_k} \cdots y_p^{a_p}.
$$
If all $a_i$ with $1 \le i \le n$ are zero, we are done.
To get a contradiction, 
assume that there is some $1 \le k \le n$ with $a_k \not= 0$.

Let $y_k^* := \mu_{(\by,C)}(y_k)$.
Again the Laurent phenomenon
yields $b_1,\ldots,b_p \in \Z$ and $\nu \in K^\times$ such that
$$
u = \nu y_1^{b_1} \cdots y_{k-1}^{b_{k-1}}(y_k^*)^{b_k}
y_{k+1}^{b_{k+1}} \cdots y_p^{b_p}.
$$
Without loss of generality let $b_k \ge 0$. (Otherwise we 
can work with $u^{-1}$ instead of $u$.)

If $b_k = 0$, we get 
$$
\la y_1^{a_1} \cdots y_k^{a_k} \cdots y_p^{a_p} 
=  \nu y_1^{b_1} \cdots y_{k-1}^{b_{k-1}}
y_{k+1}^{b_{k+1}} \cdots y_p^{b_p},
$$
where $\la,\nu \in K^\times$.
This is a contradiction, because $a_k \not= 0$ and 
$y_1,\ldots,y_m$ are algebraically independent,
and therefore Laurent monomials in
$y_1,\ldots,y_m$ are linearly independent in $\cF$.

Next, assume that $b_k > 0$.
By definition we have 
$$
y_k^* = M_1 + M_2
$$ 
with
$$
M_1 = y_k^{-1}\prod_{c_{ik}>0}y_i^{c_{ik}}
\text{\;\;\; and \;\;\;}
M_2 = y_k^{-1}\prod_{c_{ik}<0}y_i^{-c_{ik}},
$$ 
where the products run over the positive, respectively negative,
entries in the $k$th column of the matrix $C$.

Thus  we get an equality of the form
\begin{equation}\label{eq1}
u = \la M = 
\nu(y_1^{b_1} \cdots y_{k-1}^{b_{k-1}})(M_1 + M_2)^{b_k}
(y_{k+1}^{b_{k+1}} \cdots y_p^{b_p}).
\end{equation}
We know that $M_1 \not= M_2$.
(Here we use that $m > 1$ and that exchange matrices are by definition connected. Otherwise, one could get exchange relations
of the form $x_kx_k' = 1+1$.)
Thus the right-hand side of Equation~(\ref{eq1}) is
a non-trivial linear combination of 
$b_k+1 \ge 2$ pairwise different 
Laurent monomials 
in $y_1,\ldots,y_m$.
This is again a contradiction, since $y_1,\ldots,y_m$ are algebraically
independent.
\end{proof}

\begin{Cor}\label{associate}
For any seed $(\bx,B)$ of $\cF$
the following hold:
\begin{itemize}

\item[(i)] 
Let $y$ and $z$ be non-zero elements in
$\cA(\bx,B)$.
Then $y$ and $z$
are associate if and only if there
exist $a_{n+1},\ldots,a_p \in \Z$ 
and $\la \in K^\times$ with 
$$
y= \la x_{n+1}^{a_{n+1}} \cdots x_p^{a_p}z.
$$

\item[(ii)]
Let $y$ and $z$ be cluster variables of
$\cA(\bx,B)$.
Then $y$ and $z$
are associate if and only if 
$y=z$. 

\end{itemize}
\end{Cor}

\begin{proof}
Part (i) follows directly from Theorem~\ref{Thminvert}.
To prove (ii), let $\by$ and $\bz$ be clusters of $\cA(\bx,B)$.
Assume $y_i$ and $z_j$ are associate for some 
$1 \le i,j \le n$.
By (i) there are $a_{n+1},\ldots,a_p \in \Z$ and
$\la \in K^\times$ with
$y_i = \la
x_{n+1}^{a_{n+1}} \cdots x_p^{a_p}z_j$.
By Theorem~\ref{Laurent} we know that 
there exist $b_1,\ldots,b_n \in \Z$ and
a polynomial $f$ in $\Z[z_1,\ldots,z_m]$ 
with
$$
y_i = \frac{f}{z_1^{b_1} \cdots z_n^{b_n}}
$$
and $f$ is not divisible by any $z_1,\ldots,z_n$.
The polynomial $f$ and $b_1,\ldots,b_n$ are uniquely determined by
$y_i$.
It follows that $\la \in \Z$ and $a_{n+1},\ldots,a_p \ge 0$.
But we also have 
$z_j = \la^{-1} x_{n+1}^{-a_{n+1}} \cdots x_p^{-a_p}y_i$.
Reversing the role of $y_i$ and $z_j$ we get 
$-a_{n+1},\ldots,-a_p \ge 0$ and $\la^{-1} \in \Z$.
This implies $y_i = z_j$ or $-y_i = z_j$.
By the remark at the beginning of Section~\ref{introlaurent}
we know that
$z_j = f/g$ for some $f,g \in \N[y_1,\ldots,y_m]$.
Assume that $-y_i = z_j$. We get
$z_j = -y_i = f/g$ and therefore $f+y_ig = 0$.
This is a contradiction to the algebraic independence of
$y_1,\ldots,y_m$.
Thus we proved (ii).
\end{proof}

We thank Giovanni Cerulli Irelli for helping us with the final
step of the proof of
Corollary~\ref{associate}(ii).

Two clusters $\by$ and $\bz$ of a cluster algebra $\cA(\bx,B)$ are \emph{non-associate} if there are no $1 \le i,j \le n$ such that
$y_i$ and $z_j$ are associate.

\begin{Cor}\label{disass}
For clusters $\by$ and $\bz$ of $\cA(\bx,B)$ the
following are equivalent:
\begin{itemize}

\item[(i)]
The clusters
$\by$ and $\bz$ are non-associate.

\item[(ii)]
The clusters
$\by$ and $\bz$ are disjoint.

\end{itemize}
\end{Cor}

\begin{proof}
Non-associate clusters are obviously disjoint.
The converse follows directly from
Corollary~\ref{associate}(ii).
\end{proof}


\section{Irreducibility of cluster variables}\label{Sec3}


In this section we prove Theorem~\ref{Main1}(ii).
The proof is very similar to
the proof of Theorem~\ref{Thminvert}.

\begin{Thm}\label{Thmirr}
Let $(\bx,B)$ be a seed of $\cF$.
Then 
any cluster variable in $\cA(\bx,B)$ is irreducible.
\end{Thm}

\begin{proof}
Let $(\by,C)$ be any seed of $\cA(\bx,B)$.
We know from Theorem~\ref{Thminvert} that
the cluster variables of $\cA(\bx,B)$ are
non-invertible in $\cA(\bx,B)$.

Assume that $y_k$ is not irreducible for some 
$1 \le k \le n$.
Thus
$y_k = y_k'y_k''$ for some non-invertible
elements $y_k'$ and $y_k''$ in $\cA(\bx,B)$.
Since $y_k$ is invertible in $\cL_\by$, we know that
$y_k'$ and $y_k''$ are both invertible in $\cL_\by$.
Thus by Lemma~\ref{lemma1} there are $a_i,b_i \in \Z$ 
and $\la',\la'' \in K^\times$ with
$$
y_k'  = \la'y_1^{a_1} \cdots y_s^{a_s} \cdots y_p^{a_p}
\text{\;\;\; and \;\;\;}
y_k'' = \la''y_1^{b_1} \cdots y_s^{b_s} \cdots y_p^{b_p}.
$$
Since $y_k = y_k'y_k''$, we get $a_s+b_s = 0$
for all $s \not= k$ and $a_k + b_k = 1$.

Assume that $a_s = 0$ for all $ 1\le s \le n$ with $s \not= k$.
Then $y_k' = \la'y_k^{a_k}y_{n+1}^{a_{n+1}} \cdots y_p^{a_p}$ and $y_k'' = \la''y_k^{b_k}y_{n+1}^{b_{n+1}} \cdots y_p^{b_p}$.
If $a_k \le 0$, then $y_k'$ is invertible in $\cA(\bx,B)$,
and if $a_k > 0$, then $y_k''$ is invertible in $\cA(\bx,B)$.
In both cases we get a contradiction.

Next assume $a_s \not= 0$
for some $1 \le s \le n$ with $s \not= k$.
Let
$y_s^* := \mu_{(\by,C)}(y_s)$.
Thus we have
$$
y_s^* = M_1 + M_2
$$
with
$$
M_1 = y_s^{-1}\prod_{c_{is}>0} y_i^{c_{is}}
\text{\;\;\; and \;\;\;}
M_2 = y_s^{-1} \prod_{c_{is}<0} y_i^{-c_{is}},
$$
where the products run over the positive, respectively negative, 
entries in the $s$th column of the matrix $C$.

Since $s \not= k$, we see that
$y_k$ and therefore also $y_k'$ and $y_k''$ are invertible in $\cL_{\mu_s(\by,C)}$.
Thus by Lemma~\ref{lemma1}
there are $c_i,d_i \in \Z$ and $\nu',\nu'' \in K^\times$ with
$$
y_k'  = \nu'y_1^{c_1} \cdots y_{s-1}^{c_{s-1}}
(y_s^*)^{c_s}y_{s+1}^{c_{s+1}} \cdots y_p^{c_p}
\text{\;\;\; and \;\;\;}
y_k''  = \nu''y_1^{d_1} \cdots 
y_{s-1}^{d_{s-1}}(y_s^*)^{d_s}y_{s+1}^{d_{s+1}} \cdots y_p^{d_p}.
$$
Note that $c_s+d_s = 0$.
Without loss of generality we assume that $c_s \ge 0$.
(If $c_s < 0$, we continue to work with $y_k''$ instead of $y_k'$.)
If $c_s = 0$, we get
$$
y_k'  = \la'y_1^{a_1} \cdots y_s^{a_s} \cdots y_p^{a_p} =
\nu'y_1^{c_1} \cdots y_s^0 \cdots y_p^{c_p}.
$$
This is a contradiction, since $a_s \not= 0$ and $y_1,\ldots,y_m$ are
algebraically independent.
If $c_s > 0$, then 
\begin{align*}
y_k'  
&= 
\la'y_1^{a_1} \cdots y_{s-1}^{a_{s-1}} y_s^{a_s} y_{s+1}^{a_{s+1}}\cdots y_p^{a_p} \\
&=
\nu'y_1^{c_1} \cdots y_{s-1}^{c_{s-1}} (y_s^*)^{c_s} y_{s+1}^{c_{s+1}}\cdots y_p^{c_p} \\
&=
\nu'y_1^{c_1} \cdots y_{s-1}^{c_{s-1}}(M_1+M_2)^{c_s}
y_{s+1}^{c_{s+1}} \cdots y_p^{c_p}.
\end{align*}
We know that $M_1 \not= M_2$.
Thus the Laurent monomial $y_k'$ is a 
non-trivial linear combination of
$c_s+1 \ge 2$ pairwise different Laurent monomials in
$y_1,\ldots,y_m$,
a contradiction.
\end{proof}

Note that the coefficients $x_{p+1},\ldots,x_m$ of $\cA(\bx,B)$
are obviously irreducible in $\cL_\bx$.
Since $\cA(\bx,B) \subseteq \cL_\bx$,
they are also irreducible in $\cA(\bx,B)$.


\section{Factorial cluster algebras}
\label{Sec4}


\subsection{A factoriality criterion}
This section contains the proofs of Theorem~\ref{Main2}
and Corollary~\ref{Main3}.

\begin{Thm}\label{Thmfac}
Let $\by$ and $\bz$ be disjoint clusters of $\cA(\bx,B)$,
and let $U$ be a factorial subalgebra of $\cA(\bx,B)$
such that
$$
\{ y_1,\ldots,y_n,z_1,\ldots,z_n,x_{n+1}^{\pm 1},\ldots,
x_p^{\pm1},x_{p+1},\ldots,x_m \} \subset U.
$$
Then we have
$$
U = \cA(\bx,B) = U(\by,\bz).
$$
\end{Thm}

\begin{proof}
Let $u \in U(\by,\bz) = \cL_\by \cap \cL_\bz$.
Thus we have
$$
u = \frac{f}{y_1^{a_1} y_2^{a_2} \cdots y_p^{a_p}} =
\frac{g}{z_1^{b_1} z_2^{b_2} \cdots z_p^{b_p}},
$$
where $f$ is a polynomial in $y_1,\ldots,y_m$, and
$g$ is a polynomial in $z_1,\ldots,z_m$, and
$a_i,b_i \ge 0$ for all $1 \le i \le p$.
By the Laurent phenomenon 
it is enough to show that $u \in U$.

Since $y_i,z_i \in U$ for all $1 \le i \le m$, we get
the identity
$$
f z_1^{b_1} z_2^{b_2} \cdots z_n^{b_n}z_{n+1}^{b_{n+1}} \cdots z_p^{b_p} =
g y_1^{a_1} y_2^{a_2} \cdots y_n^{a_n}y_{n+1}^{a_{n+1}} \cdots y_p^{a_p}
$$
in $U$.

By Theorem~\ref{Thmirr} the cluster variables $y_i$ and $z_i$ 
with $1 \le i \le n$ are
irreducible in $\cA(\bx,B)$.
In particular, they are irreducible in the subalgebra
$U$ of $\cA(\bx,B)$.
The elements $y_{n+1}^{a_{n+1}} \cdots y_p^{a_p}$ and $z_{n+1}^{b_{n+1}} \cdots z_p^{b_p}$ are units in $U$.
(Recall that $x_i = y_i = z_i$ for all $n+1 \le i \le m$.)

The clusters $\by$ and $\bz$ are disjoint.
Now Corollary~\ref{disass} implies that
the elements $y_i$ and $z_j$
are non-associate for all $1 \le i,j \le n$.
Thus,
by the factoriality of $U$, the monomial $y_1^{a_1} y_2^{a_2} \cdots y_n^{a_n}$ divides $f$ in $U$.
In other words there is some $h \in U$ with
$f = h y_1^{a_1} y_2^{a_2} \cdots y_n^{a_n}$.
It follows that 
$$
u =  \frac{f}{y_1^{a_1} y_2^{a_2} \cdots y_p^{a_p}} =
\frac{h y_1^{a_1} y_2^{a_2} \cdots y_n^{a_n}}{y_1^{a_1} y_2^{a_2} \cdots y_p^{a_p}} = 
\frac{h}{y_{n+1}^{a_{n+1}} \cdots y_p^{a_p}} =
h y_{n+1}^{-a_{n+1}} \cdots y_p^{-a_p}.
$$
Since $h \in U$ and $y_{n+1}^{\pm 1},\ldots,y_p^{\pm 1} \in U$,
we get $u \in U$.
This finishes the proof.
\end{proof}

\begin{Cor}
Assume that $\cA(\bx,B)$ is factorial.
\begin{itemize}
\setlength{\itemsep}{9pt}

\item[(i)]
If $\by$ and $\bz$ are
disjoint clusters of $\cA(\bx,B)$, then 
$\cA(\bx,B) = U(\by,\bz)$.

\item[(ii)]
For any $(\by,C) \sim (\bx,B)$ we have
$\cA(\bx,B) = U(\by,C)$.

\end{itemize}
\end{Cor}

\begin{proof}
Part (i) follows directly from Theorem~\ref{Main2}.
To prove part (ii), assume $(\by,C) \sim (\bx,B)$ and
let $u \in U(\by,C)$. 
For $1 \le k \le n$ let
$(\by_k,C_k) := \mu_k(\by,C)$
and $y_k^* := \mu_{(\by,C)}(y_k)$.
We get
$$
u = \frac{f}{y_1^{a_1} \cdots y_k^{a_k} \cdots y_p^{a_p}}  
=
 \frac{f_k}{y_1^{b_1} \cdots (y_k^*)^{b_k}  \cdots y_p^{b_p}}  
$$
for a polynomial $f$ in $y_1,\ldots,y_k,\ldots,y_m$,
a polynomial $f_k$ in $y_1,\ldots,y_k^*,\ldots,y_m$,
and $a_i,b_i \ge 0$.
This yields an equality
\begin{equation}\label{eq2}
fy_1^{b_1} \cdots (y_k^*)^{b_k} \cdots y_p^{b_p} =
f_k y_1^{a_1} \cdots y_k^{a_k} \cdots y_p^{a_p}
\end{equation}
in $\cA(\bx,B)$.
Now we argue similarly as in the proof of Theorem~\ref{Thmfac}.
The cluster variables
$y_1,\ldots,y_n,y_1^*,\ldots,y_n^*$ are obviously
pairwise different.
Now Corollary~\ref{associate}(ii) implies that they are
pairwise non-associate, and by
Theorem~\ref{Thmirr} they are irreducible in $\cA(\bx,B)$.
Thus
by the factoriality of $\cA(\bx,B)$,
Equation~(\ref{eq2}) implies that 
$y_k^{a_k}$ divides $f$ in $\cA(\bx,B)$.
Since this holds for all $1 \le k \le n$, we get that
$y_1^{a_1} \cdots y_k^{a_k} \cdots y_n^{a_n}$ divides $f$ in $\cA(\bx,B)$.
It follows that $u \in \cA(\bx,B)$.
\end{proof}

\subsection{Existence of disjoint clusters}
One assumption of Theorem~\ref{Thmfac} is the existence
of disjoint clusters in $\cA(\bx,B)$.
We can prove this under a mild assumption.
But it should be true in general.

\begin{Prop}
Assume that the cluster monomials of $\cA(\bx,B)$ 
are linearly independent.
Let $(\by,C)$ be a seed of $\cA(\bx,B)$, and let
$$
(\bz,D) := \mu_n \cdots \mu_2\mu_1(\by,C).
$$
Then the clusters $\by$ and $\bz$ are disjoint.
\end{Prop}

\begin{proof}
Set $(\by[0],C[0]) := (\by,C)$, and for $1 \le k \le n$ let
$(\by[k],C[k]) := \mu_k(\by[k-1],C[k-1])$ and
$(y_1[k],\ldots,y_m[k]) := \by[k]$.
We claim that
$$
\{ y_1[k],\ldots,y_k[k] \} \cap \{ y_1,\ldots, y_n \} = \varnothing.
$$
For $k=1$ this is straightforward.
Thus let $k \ge 2$, and 
assume that our claim is true for $k-1$.
To get a contradiction, 
assume that $y_k[k] = y_j$ for some
$1 \le j \le n$.
(By the induction assumption we know that
$\{ y_1[k],\ldots,y_{k-1}[k] \} \cap \{ y_1,\ldots, y_n \} = \varnothing$, since $y_i[k] = y_i[k-1]$ for all 
$1 \le i \le k-1$.)

We have 
$\by[k] = (y_1[k],\ldots,y_k[k],y_{k+1},\ldots,y_m)$.
Since $y_1[k],\ldots,y_k[k],y_{k+1},\ldots,y_m$ are
algebraically independent and $y_k[k] \not= y_k$, we get 
$1 \le j \le k-1$.
Since $(\by[j],C[j]) = \mu_j(\by[j-1],C[j-1])$,
it follows that $(y_j[j-1],y_j[j])$ is an exchange pair
of $\cA(\bx,B)$.
Next, observe that
$y_k[k] = y_j = y_j[j-1]$ and 
$y_j[k] = y_j[j]$.
Thus $y_j[j-1]$ and  $y_j[j]$ are both
contained in $\{ y_1[k],\ldots,y_m[k] \}$, and therefore
$y_j[j-1]y_j[j]$ is a cluster monomial. 
The
corresponding exchange relation gives a contradiction
to the linear independence of cluster monomials.
\end{proof}

Fomin and Zelevinsky
\cite[Conjecture~4.16]{FZ3} conjecture that the cluster monomials of $\cA(\bx,B)$
are always linearly independent.
Under the assumptions that
$B$ has maximal rank and that
$B^\circ$ is skew-symmetric, the conjecture follows from \cite[Theorem~1.7]{DWZ}.


\section{The divisibility group of a cluster algebra}\label{Sec5}


Let $R$ be an integral domain, and let $\Quot(R)$
be the field of fractions of $R$.
Set $\Quot(R)^* := \Quot(R) \setminus \{0\}$.
The abelian group 
$$
G(R) := (\Quot(R)^*/R^\times,\cdot)
$$ 
is the \emph{divisibility group} of $R$.

For $g,h \in \Quot(R)^*$ let
$g \le h$ provided $hg^{-1} \in R$.
This relation is reflexive and transitive and it induces
a partial ordering on $G(R)$.

Let $I$ be a set.
The abelian group $(\Z^{(I)},+)$ is equipped with the following partial ordering:
We set $(x_i)_{i \in I} \le (y_i)_{i \in I}$ if $x_i \le y_i$
for all $i$.
(By definition, the elements in $\Z^{(I)}$ are
tuples $(x_i)_{i \in I}$ of integers $x_i$ such that
only finitely many $x_i$ are non-zero.)

There is the following well-known criterion for
the factoriality of $R$, see for example \cite[Section~2]{C}.

\begin{Prop}
For an integral domain $R$ the following are equivalent:
\begin{itemize}
\setlength{\itemsep}{9pt}

\item[(i)]
$R$ is factorial.

\item[(ii)]
There is a set $I$ and
a group isomorphism 
$$
\phi\df G(R) \to \Z^{(I)}
$$
such that
for all $g,h \in G(R)$ we have
$g \le h$ if and only if $\phi(g) \le \phi(h)$.

\end{itemize}
\end{Prop}

Not all cluster algebras $\cA(\bx,B)$ are factorial, but at least
one part of the above factoriality criterion is satisfied:

\begin{Prop}
For any seed $(\bx,B)$ of $\cF$
the divisibility group 
$G(\cA(\bx,B))$
is isomorphic to $\Z^{(I)}$, where
$$
I := \left\{ f \in K[x_1,\ldots,x_m] \mid f \text{ is irreducible and } f \not= x_i \text{ for } n+1 \le i \le p \right\}/K^\times
$$
is the set of irreducible polynomials unequal to any $x_{n+1},\ldots,x_p$ in
$K[x_1,\ldots,x_m]$ up to non-zero scalar multiples.
\end{Prop}

\begin{proof}
By the Laurent phenomenon and the definition
of a seed we get
$$
\Quot(\cA(\bx,B)) = \Quot(\cL_\bx) = 
K(x_1,\ldots,x_m).
$$
Furthermore, by Theorem~\ref{Thminvert} we have
$$
\cA(\bx,B)^\times = \{\la x_{n+1}^{a_{n+1}} \cdots x_p^{a_p} \mid \la \in K^\times, a_i \in \Z \}.
$$
Any element
in $K(x_1,\ldots,x_m)$ is of the form $f_1 \cdots f_s
g_1^{-1} \cdots g_t^{-1}$ with
$f_i,g_j$ irreducible in $K[x_1,\ldots,x_m]$.
Using that the polynomial ring
$K[x_1,\ldots,x_m]$ is factorial, and
working modulo $\cA(\bx,B)^\times$ yields the result.
\end{proof}


\section{Examples of non-factorial cluster algebras}\label{Sec6.1}


\subsection{}\label{nonfac3}
For a matrix $A \in M_{m,n}(\Z)$ and $1 \le i \le n$ let
$c_i(A)$ be the $i$th column of $A$.

\begin{Prop}\label{nonfac1}
Let $(\bx,B)$ be a seed of $\cF$.
Assume
that $c_k(B) = c_s(B)$ or $c_k(B) = -c_s(B)$ for some
$k \not= s$ with $b_{ks} = 0$. 
Then $\cA(\bx,B)$ is
not factorial.
\end{Prop}

\begin{proof}
Define  $(\by,C) := \mu_k(\bx,B)$ and
$(\bz,D) := \mu_s(\by,C)$.
We get
$$
y_k = z_k = x_k^{-1}(M_1+M_2),
$$
where
$$
M_1 := \prod_{b_{ik}>0} x_i^{b_{ik}}
\text{\;\;\; and \;\;\;}
M_2 := \prod_{b_{ik}<0}x_i^{-b_{ik}}.
$$
By the mutation rule, we have $c_k(C) = -c_k(B)$, and
since $b_{ks}=0$, we get $c_s(C) = c_s(B)$.
Since $c_s(B) = c_k(B)$ or $c_s(B) = -c_k(B)$, this implies 
that 
$$
z_s = x_{s}^{-1}(M_1+M_2).
$$
The cluster variables $x_k,x_s,z_k,z_s$ are pairwise
different.
Thus they are pairwise non-associate by
Corollary~\ref{associate}(ii), 
and by Theorem~\ref{Thmirr} they are irreducible
in $\cA(\bx,B)$.
Obviously, we have
$$
x_kz_k = x_sz_s.
$$
Thus $\cA(\bx,B)$ is not factorial.
\end{proof}

To give a concrete example of a cluster algebra, which is
not factorial,
assume $m=n=p=3$, and let
$B \in M_{m,n}(\Z)$ be the matrix
$$
B = 
\left(\bbm
0&-1&0\\
1&0&-1\\
0&1&0
\ebm
\right).
$$
The matrix $B$ obviously satisfies the assumptions of
Proposition~\ref{nonfac1}.
Note that $B = B^\circ$ is skew-symmetric, and that
$\Gamma(B)$ is the quiver
$$
\xymatrix{
3 \ar[r] & 2 \ar[r] &1.
}
$$
Thus $\cA(\bx,B)$ is a cluster algebra of Dynkin type $\A_3$.
(Cluster algebras with finitely many cluster variables
are classified via Dynkin types, for details see \cite{FZ2}.)

Define 
$(\bz,D) := \mu_3\mu_1(\bx,B)$.
We get
$z_1 = x_1^{-1}(1+x_2)$,
$z_3 = x_3^{-1}(1+x_2)$ and therefore
$x_1z_1 = x_3z_3$.

Clearly, the cluster variables $x_1,x_3,z_1,z_3$ are pairwise different.
Using Corollary~\ref{associate}(ii) we get that
$x_1,x_3,z_1,z_3$ are pairwise non-associate, and by Theorem~\ref{Thmirr} they are irreducible.
Thus
$\cA(\bx,B)$ is not factorial.

\subsection{}\label{nonfac2}
The next example is due to Philipp Lampe.
It gives a negative answer to Zelevinsky's Question~\ref{question1}.

\begin{Prop}[{{\cite{La}}}]\label{lampe}
Let $K=\C$, $m=n=2$ and 
$$
B = \left(\bbm
0 & -2\\
2 & 0
\ebm\right).
$$
Then $\cA(\bx,B)$ is not factorial.
\end{Prop}

The proof of the following result is
a straightforward generalization of 
Lampe's proof of Proposition~\ref{lampe}.

\begin{Prop}\label{nonfac4}
Let $(\bx,B)$ be a seed of $\cF$.
Assume that there exists some $1 \le k \le n$ such that
the polynomial  $X^d+Y^d$ is not irreducible in $K[X,Y]$,
where $d := \gcd(b_{1k},\ldots,b_{mk})$ is the greatest common divisor of $b_{1k},\ldots,b_{mk}$.
Then $\cA(\bx,B)$ is not factorial.
\end{Prop}

\begin{proof}
Let $X^d+Y^d = f_1 \cdots f_t$, where the $f_j$ are irreducible polynomials in $K[X,Y]$.
Since $X^d+Y^d$ is not irreducible in $K[X,Y]$, we have $t \ge 2$.
Let $y_k := \mu_{(\bx,B)}(x_k)$.
The corresponding exchange relation is
$$
x_ky_k = \prod_{b_{ik}>0} x_i^{b_{ik}} +  
\prod_{b_{ik}<0} x_i^{-b_{ik}} = M^d+N^d= \prod_{j=1}^tf_j(M,N),
$$
where
$$
M:=  \prod_{b_{ik}>0} x_i^{b_{ik}/d}
\text{\;\;\; and \;\;\;}
N:=  \prod_{b_{ik}<0} x_i^{-b_{ik}/d}.
$$
Clearly, each $f_j(M,N)$ is contained in $\cA(\bx,B)$.
To get a contradiction, assume that $\cA(\bx,B)$ is factorial.
By Theorem~\ref{Thminvert} none of the elements
$f_j(M,N)$ is invertible in $\cA(\bx,B)$.
Since $\cA(\bx,B)$ is factorial, each $f_j(M,N)$ is equal to a product
$f_{1j} \cdots f_{a_jj}$, where the $f_{ij}$ are irreducible in
$\cA(\bx,B)$ and $a_j \ge 1$.
By Theorem~\ref{Thmirr} the cluster variables $x_k$ and
$y_k$ are irreducible in $\cA(\bx,B)$.
It follows that $a_1+ \cdots + a_t = 2$, since
$\cA(\bx,B)$ is factorial.
This implies $t=2$ and
$a_1=a_2=1$.
In particular, $f_1(M,N)$ and $f_2(M,N)$ are irreducible in $\cA(\bx,B)$, and we have $x_ky_k = f_1(M,N)f_2(M,N)$.
For $j=1,2$ the elements $x_k$ and $f_j(M,N)$ cannot be associate, since $f_j(M,N)$ is just a $K$-linear combination
of monomials in $\{ x_1,\ldots,x_m \} \setminus \{ x_k\}$. (Here we use Corollary~\ref{associate}(i) and the fact that $b_{kk}=0$.)
This is a contradiction to the factoriality of $\cA(\bx,B)$. 
\end{proof}

Note that a polynomial of the form $X^d+Y^d$ is
irreducible if and only if $X^d+1$ is irreducible.

\begin{Cor}
Let $K=\C$, $m=n=2$ and 
$$
B = \left(\bbm
0 & -c\\
d & 0
\ebm\right)
$$
with $c \ge 1$ and $d \ge 2$.
Then $\cA(\bx,B)$ is not factorial.
\end{Cor}

\begin{proof}
For $k=1$ the assumptions of Proposition~\ref{nonfac4}
hold.
(We have $\gcd(0,d) = d$, and 
the polynomial $X^d+ 1$
is not irreducible in $\C[X]$.)
\end{proof}

\begin{Cor}
Let $m=n=2$ and
$$
B = \left(\bbm
0 & -c\\
d & 0
\ebm\right)
$$
with $c \ge 1$ and $d \ge 3$ an odd number.
Then $\cA(\bx,B)$ is not factorial.
\end{Cor}

\begin{proof}
For $k=1$ the assumptions of Proposition~\ref{nonfac4}
hold. 
(We have $\gcd(0,d) = d$, and for
odd $d$ 
we have 
$$
X^d+1 = (X+1)\left(\sum_{j=0}^{d-1} (-1)^jX^j\right).
$$
Thus $X^d+1$ is not irreducible in $K[X]$.)
\end{proof}


\section{Examples of factorial cluster algebras}\label{Sec6.2}


\subsection{Cluster algebras of Dynkin type $\A$
as polynomial rings}
\label{linear}
Assume $m=n+1=p+1$, and let
$B \in M_{m,n}(\Z)$ be the matrix
$$
B = 
\left(\bbm
0&-1&&&&\\
1&0&-1&&&\\
&1&0&\ddots&&\\
&&1&\ddots&-1&\\
&&&\ddots&0&-1\\
&&&&1&0\\\hline
&&&&&1
\ebm
\right).
$$
Obviously, $B^\circ$ is skew-symmetric,
$\Gamma(B)$ is the quiver
$$
\xymatrix{
m \ar[r] & \cdots \ar[r] & 2 \ar[r] & 1,
}
$$
and $\cA(\bx,B)$ is a cluster algebra of Dynkin type
$\A_n$.
Note that $\cA(\bx,B)$ has exactly one coefficient, and that
this coefficient is non-invertible.

Let $(\bx[0],B[0]) := (\bx,B)$.
For each $1 \le i \le m-1$
we define inductively a seed by 
$$
(\bx[i],B[i]) := 
\mu_{m-i} \cdots \mu_2 \mu_1(\bx[i-1],B[i-1]).
$$ 
For $0 \le i \le m-1$ set
$(x_1[i],\ldots,x_m[i]) := \bx[i]$.

For simplicity we define $x_0[i] := 1$ and
$x_{-1}[i] := 0$ for all $i$.

\begin{Lem}\label{linear1}
For $0 \le i \le m-2$, 
$1 \le k \le m-1-i$ and $0 \le j \le i$ we have
\begin{equation}\label{eq3}
\mu_{(\bx[i],B[i])}(x_k[i]) = 
\frac{x_{k-1}[i]+x_{k+1}[i]}{x_k[i]} = 
\frac{x_{k-1+j}[i-j]+x_{k+1+j}[i-j]}{x_{k+j}[i-j]}. 
\end{equation}
\end{Lem}

\begin{proof}
The first equality follows from the definition of
$(\bx[i],B[i])$ and the mutation rule.
The second equality is proved by induction on $i$.
\end{proof}

\begin{Cor}\label{linear2}
For $0 \le i \le m-2$ we have
\begin{align}
\label{eq5}
x_{i+2} &= x_1[i+1]x_{i+1}-x_i,\\ 
\label{eq6}
x_{i+1}[1] &= x_1[i+1]x_i[1] - x_{i-1}[1]. 
\end{align}
\end{Cor}

\begin{proof}
Equation~(\ref{eq5}) follows from (\ref{eq3}) for $k=1$ and
$j=i$.
The case $k=1$ and $j=i-1$ yields
Equation~(\ref{eq6}).
\end{proof}

\begin{Prop}\label{linear3}
The elements
$x_1[0],x_1[1],\ldots,x_1[m-1]$ are algebraically independent and
$$
K[x_1[0],x_1[1],\ldots,x_1[m-1]] = \cA(\bx,B).
$$ 
In particular, the cluster algebra
$\cA(\bx,B)$
is a polynomial ring in
$m$ variables.
\end{Prop}

\begin{proof}
It follows from Equation~(\ref{eq5}) that
$$
x_1[i] \in K(x_1,\ldots,x_{i+1}) \setminus K(x_1,\ldots,x_i)
$$
for all $1 \le i \le m-1$.
Since
$x_1,\ldots,x_m$ are algebraically independent,
this implies that
$x_1[0],x_1[1],\ldots,x_1[m-1]$ are algebraically
independent as well.
Thus 
$$
U := K[x_1[0],x_1[1],\ldots,x_1[m-1]]
$$ is a polynomial ring in $m$
variables.
In particular, $U$ is factorial.
Equation~(\ref{eq5}) implies that
$x_1,\ldots,x_m \in U$, and
Equation~(\ref{eq6}) yields that
$x_1[1],\ldots,x_m[1] \in U$.
Clearly, the clusters $\bx$ and $\bx[1]$ are disjoint.
Thus the assumptions of Theorem~\ref{Thmfac} 
are satisfied, and
we get
$U = \cA(\bx,B)$.
\end{proof}

The cluster algebra $\cA(\bx,B)$ as defined
above has been studied by several people.
It is related to a $T$-system of Dynkin type 
$\A_1$
with a certain boundary condition, see \cite{DK}. 
Furthermore, for $K=\C$ the cluster algebra
$\cA(\bx,B)$ is naturally isomorphic to the complexified 
Grothendieck ring  of the category $\cC_n$ of finite-dimensional
modules of level $n$ over the quantum loop algebra
of Dynkin type $\A_1$, see \cite{HL,N2}. 
It is well known, that $\cA(\bx,B)$ is a polynomial ring.
We just wanted to demonstrate how to use
Theorem~\ref{Thmfac} in practise.

\subsection{Acyclic cluster algebras as polynomial rings}\label{acyclic}
Let $C = (c_{ij}) \in M_{n,n}(\Z)$ be a 
\emph{generalized Cartan matrix}, i.e.
$C$ is symmetrizable, $c_{ii} = 2$ for
all $i$ and $c_{ij} \le 0$ for all $i \not= j$.

Assume that $m=2n=2p$, and let $(\bx,B)$ be a seed of $\cF$,
where
$B = (b_{ij}) \in M_{2n,n}(\Z)$ is defined as
follows:
For $1 \le i \le 2n$ and $1 \le j \le n$ let
$$
b_{ij} :=
\begin{cases}
0 & \text{if $i=j$},\\
-c_{ij} & \text{if $1 \le i < j \le n$},\\
c_{ij} & \text{if $1 \le j < i \le n$},\\
1     & \text{if $i = n+j$},\\
c_{i-n,j} & \text{if $n+1 \le i \le 2n$ and $i-n<j$},\\
0 & \text{if $n+1 \le i \le 2n$ and $i-n>j$}. 
\end{cases}
$$
Thus we have 
$$
B =
\left(\bbm
0 &b_{12} & b_{13} & \cdots & b_{1n} \\
b_{21} & 0 & b_{23} & \cdots & b_{2n} \\
b_{31} & b_{32} & 0 & \ddots & \vdots \\
\vdots & \vdots & \ddots & \ddots & b_{n-1,n} \\
b_{n1} &b_{n2} & \cdots & b_{n,n-1} & 0 \\\hline
1 &-b_{12} & -b_{13} & \cdots & -b_{1n} \\
0 & 1 & -b_{23} & \cdots & -b_{2n} \\
0 & 0 & 1 & \ddots & \vdots \\
\vdots & \vdots & \ddots & \ddots & -b_{n-1,n} \\
0 & 0 & 0 & 0 & 1 
\ebm
\right).
$$
Clearly, $(\bx,B)$ is an acyclic seed.
Namely, if $i \to j$ is an arrow in $\Sigma(B)$, then
$i < j$.
Up to simultaneous reordering of columns and rows, each
acyclic skew-symmetrizable matrix in $M_{n,n}(\Z)$ is of
the form $B^\circ$ with $B$ defined as above.
Note that $\cA(\bx,B)$ has exactly $n$ coefficients, and that
all these coefficients are non-invertible.

For $1 \le i \le n$ let
$$
(\bx[1],B[1]) :=\mu_n \cdots \mu_2\mu_1(\bx,B)
$$
and $(x_1[1],\ldots,x_{2n}[1]) := \bx[1]$.
Let $B_0 := B$, and for $1 \le i \le n$ let
$B_i := \mu_i(B_{i-1})$.
Thus we have $B_n = B[1]$.
It is easy to work out the matrices $B_i$ explicitly:
The matrix $B_i$ 
is obtained from $B_{i-1}$ by changing
the sign in the $i$th row and the $i$th column of the 
principal part $B_{i-1}^\circ$.
Furthermore, the $(n+i)$th row 
$$
(0,\ldots,0,1,-b_{i,i+1},-b_{i,i+2},\ldots,-b_{in})
$$ 
of $B_{i-1}$ gets replaced by
$$
(-b_{i1},-b_{i2},\ldots,-b_{i,i-1},-1,0,\ldots,0).
$$

If we write $N_+$ (resp. $N_-$) for the upper (resp. lower) triangular part of $B^\circ$,
we get
$$
B = \left(\bbm
1 & & N_+\\
&\ddots &\\
 N_- & & 1\\\hline
  1& &-N_+\\
 & \ddots &\\
0&&1 
\ebm\right)
\text{\;\;\; and \;\;\;}
B[1] = \left(\bbm
1 & & N_+\\
&\ddots &\\
 N_- & & 1\\\hline
  -1& &0\\
 & \ddots &\\
-N_-&&-1 
\ebm\right).
$$
In particular, the principal part $B^\circ$ of $B$
is equal to the principal part $B[1]^\circ$ of $B[1]$.

Now the definition of seed mutation yields
\begin{equation}\label{eq7}
x_k[1] = x_k^{-1}\left(
x_{n+k} + \prod_{i=1}^{k-1}x_i[1]^{b_{ik}}
 \prod_{i=k+1}^n x_i^{-b_{ik}}
\right)
\end{equation}
for $1 \le k \le n$.

\begin{Prop}\label{acyclic1}
The elements
$x_1,\ldots,x_n,x_1[1],\ldots,x_n[1]$
are algebraically independent and
$$
K[x_1,\ldots,x_n,x_1[1],\ldots,x_n[1]] = \cA(\bx,B).
$$ 
In particular, the cluster algebra
$\cA(\bx,B)$
is a polynomial ring in
$2n$ variables.
\end{Prop}

\begin{proof}
By Equation~(\ref{eq7}) and induction we have
$$
x_k[1] \in K(x_1,\ldots,x_{n+k}) \setminus K(x_1,\ldots,x_{n+k-1})
$$
for all $1 \le k \le n$.
It follows that
$x_1,\ldots,x_n,x_1[1],\ldots,x_n[1]$
are algebraically independent, and that the clusters
$\bx$ and $\bx[1]$ are disjoint.
Let 
$$
U := K[x_1,\ldots,x_n,x_1[1],\ldots,x_n[1]].
$$
Thus $U$ is a polynomial ring in $2n$ variables.
In particular, $U$ is factorial.
It follows from Equation~(\ref{eq7})
that
\begin{equation}\label{eq11}
x_{n+k} = x_k[1]x_k - 
\prod_{i=1}^{k-1} x_i[1]^{b_{ik}}
\prod_{i=k+1}^n x_i^{-b_{ik}}.
\end{equation}
This implies 
$x_{n+k} \in U$ for all
$1 \le k \le n$.
Thus the assumptions of Theorem~\ref{Thmfac} are satisfied,
and we can conclude that
$U = \cA(\bx,B)$.
\end{proof}

Proposition~\ref{acyclic1} is a special case of a much more
general result proved in \cite{GLSKM}.
But the proof presented here is new and more
elementary.

Next, we compare the basis 
$$
\cP_{\rm GLS} := \left\{
x[\ba] := x_1^{a_1} \cdots x_n^{a_n}x_1[1]^{a_{n+1}} \cdots 
x_n[1]^{a_{2n}} \mid
\ba = (a_1,\ldots,a_{2n}) \in \N^{2n}
\right\}
$$
of $\cA(\bx,B)$ resulting from Proposition~\ref{acyclic1} 
with a basis constructed by Berenstein,
Fomin and Zelevinsky \cite{BFZ}.
For $1 \le k \le n$ let 
\begin{equation}\label{eq10}
x_k' := \mu_{(\bx,B)}(x_k) = x_k^{-1} 
\left(x_{n+k}\prod_{i=1}^{k-1}x_i^{b_{ik}} +
\prod_{i=k+1}^n x_i^{-b_{ik}}
\prod_{i=1}^{k-1}x_{n+i}^{b_{ik}}\right)
\end{equation}
and set
\begin{multline*}
\cP_{\rm BFZ} := 
\{ x'[\ba] := x_1^{a_1} \cdots x_n^{a_n}x_{n+1}^{a_{n+1}}
\cdots x_{2n}^{a_{2n}}
(x_1')^{a_{2n+1}} \cdots 
(x_n')^{a_{3n}} \mid \\
\ba = (a_1,\ldots,a_{3n}) \in \N^{3n},
a_ka_{2n+k} = 0
\text{ for } 1 \le k \le n \}.
\end{multline*}

\begin{Prop}[{{\cite[Corollary~1.21]{BFZ}}}]\label{acyclic2}
The set $\cP_{\rm BFZ}$ is a basis of $\cA(\bx,B)$.
\end{Prop}

Note that the basis $\cP_{\rm GLS}$ is constructed by using cluster variables from two
seeds, namely $(\bx,B)$ and $\mu_n \cdots \mu_1(\bx,B)$,
whereas $\cP_{\rm BFZ}$ uses cluster variables from $n+1$
seeds, namely $(\bx,B)$ and
$\mu_k(\bx,B)$, where $1 \le k \le n$.

Now we insert Equation~(\ref{eq11}) into 
Equation~(\ref{eq10}) and obtain
\begin{multline}\label{eq12}
x_kx_k' = 
\left(x_k[1]x_k - 
\prod_{i=1}^{k-1} x_i[1]^{b_{ik}}
\prod_{i=k+1}^n x_i^{-b_{ik}}\right)
\prod_{i=1}^{k-1}x_i^{b_{ik}} + \\
\prod_{i=k+1}^n x_i^{-b_{ik}}
\prod_{i=1}^{k-1}
\left(x_i[1]x_i - 
\prod_{j=1}^{i-1} x_j[1]^{b_{ji}}
\prod_{j=i+1}^n x_j^{-b_{ji}}\right)^{b_{ik}}.
\end{multline}
Then we observe that the right-hand side
of Equation~(\ref{eq12}) is divisible by $x_k$
and that $x_k'$ is a polynomial in 
$x_1,\ldots,x_n,x_1[1],\ldots,x_n[1]$.
Thus we can express every element of
the basis $\cP_{\rm BFZ}$ explicitely as a linear combination of
vectors from the basis $\cP_{\rm GLS}$.

One could use Equation~(\ref{eq12}) to get an alternative
proof of Proposition~\ref{acyclic1} as pointed out by
Zelevinsky \cite{Z}.
Vice versa, using Propostion~\ref{acyclic1} yields
another proof that $\cP_{\rm BFZ}$ is a basis.

As an illustration,
for $n= 3$ the matrices $B_i$ look
as follows:
\begin{align*}
B_0 &=
\left(\bbm
0 &b_{12} & b_{13}  \\
b_{21} & 0 & b_{23}   \\
b_{31} & b_{32} & 0  \\\hline
1 &-b_{12} & -b_{13} &  \\
0 & 1 & -b_{23} &  \\
0 & 0 & 1 
\ebm
\right),
&
B_1 &=
\left(\bbm
0           &-b_{12} & -b_{13}   \\
-b_{21} & 0          & b_{23}   \\
-b_{31} & b_{32} & 0  \\\hline
-1         &0 & 0 \\
0           & 1 & -b_{23}   \\
0           & 0 & 1 
\ebm
\right),\\
B_2 &=
\left(\bbm
0           &b_{12} & -b_{13}   \\
b_{21} & 0          & -b_{23}   \\
-b_{31} & -b_{32} & 0   \\\hline
-1         &0 & 0  \\
-b_{21}           & -1 & 0  \\
0           & 0 & 1 
\ebm
\right),
&
B_3 &=
\left(\bbm
0           &b_{12}   & b_{13}   \\
b_{21}   & 0          & b_{23}  \\
b_{31}   & b_{32}  & 0   \\\hline
-1         &0  & 0  \\
-b_{21}           & -1       & 0  \\
-b_{31}           & -b_{32}         & -1 
\ebm
\right).
\end{align*}
For example, for 
$$
B = B_0 =
\left(\bbm
0           &2   & 0   \\
-2   & 0          & 1  \\
0   & -1  & 0   \\\hline
1         &-2  & 0  \\
0           & 1       & -1  \\
  0        &  0        & 1 
\ebm
\right)
$$
the quivers $\GG(B_0)$ and $\GG(B_3)$ look as follows:
$$
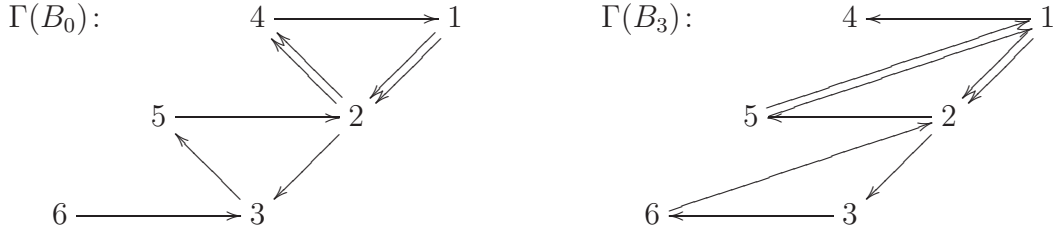
\xymatrix@-2.7pc@!{
\GG(B_0)\colon&& 4 \ar[rr]&& 1 \ar@<0.3ex>[dl]\ar@<-0.3ex>[dl]
&&\GG(B_3)\colon&& 4 && 1\ar[ll] \ar@<0.3ex>[dl]\ar@<-0.3ex>[dl]
\\
&5 \ar[rr]
& & 2 \ar@<0.3ex>[ul]\ar@<-0.3ex>[ul]\ar[dl] &
&&&5 \ar@<0.3ex>[urrr]\ar@<-0.3ex>[urrr]
& & 2\ar[ll] \ar[dl]
\\
6 \ar[rr]&& 3 \ar[ul]&&
&&6 \ar[urrr]&& 3 \ar[ll]
}
$$
The cluster algebra $\cA(\bx,B)$ is a polynomial ring
in the 6 variables
$x_1,x_2,x_3$ and
\begin{align*}
x_1[1] &= \frac{x_2^2 + x_4}{x_1},\\
x_2[1] &= \frac{x_2^4x_3 + 2x_2^2x_3x_4 + x_3x_4^2 + x_1^2x_5}{x_1^2x_2},\\
x_3[1] &= \frac{x_2^4x_3 + 2x_2^2x_3x_4 + x_3x_4^2 + x_1^2x_5 + x_1^2x_2x_6}{x_1^2x_2x_3}.
\end{align*}

\subsection{Cluster algebras arising in Lie theory 
as polynomial rings}\label{lie}
The next class of examples can be seen as a fusion of
the examples discussed in Sections~\ref{linear} and \ref{acyclic}.
In the following we use the same notation as in \cite{GLSKM}.

Let $C \in M_{n,n}(\Z)$ be a symmetric generalized Cartan matrix,
and let
$\g$ be the associated Kac-Moody Lie algebra
over $K = \C$ with triangular decomposition 
$\g = \n_- \oplus \h \oplus \n$, see \cite{K}.

Let $U(\n)_{\rm gr}^*$ be the graded dual of the enveloping
algebra $U(\n)$ of $\n$.
To each element $w$ in the Weyl group $W$ of $\g$ 
one can associate a subalgebra 
$\cR(\cC_w)$ of $U(\n)_{\rm gr}^*$ and a cluster 
algebra $\cA(\cC_w)$, see \cite{GLSKM}.
Here $\cC_w$ denotes a Frobenius category associated to $w$,
see \cite{BIRS,GLSKM}.

In \cite{GLSKM} we constructed a natural algebra isomorphism
$$
\cA(\cC_w) \to \cR(\cC_w).
$$
This yields a cluster algebra structure on $\cR(\cC_w)$.

Let $\bi = (i_r,\ldots,i_1)$ be a reduced expression of $w$.
In \cite{GLSKM} we studied
two cluster-tilting modules $V_\bi = V_1 \oplus \cdots \oplus V_r$ and $T_\bi = T_1 \oplus \cdots \oplus T_r$ in $\cC_w$,
which are associated to $\bi$.
These modules yield two
disjoint clusters 
$(\delta_{V_1},\ldots,\delta_{V_r})$ and
$(\delta_{T_1},\ldots,\delta_{T_r})$ of $\cR(\cC_w)$.
The exchanges matrices are of size $r \times (r-n)$.
In contrast to our conventions in this article, the $n$ 
coefficients are $\delta_{V_k} = \delta_{T_k}$ with $k^+ = r+1$,
where $k^+$ is defined as in \cite{GLSKM}, and none of these
coefficients is invertible.
Furthermore, we studied a module
$M_\bi = M_1 \oplus \cdots \oplus M_r$ in $\cC_w$, which
yields
cluster variables
$\delta_{M_1},\ldots,\delta_{M_r}$ of $\cR(\cC_w)$. 
(These do not form a cluster.)
Using methods from Lie theory we obtained the following
result.

\begin{Thm}[{{\cite[Theorem~3.2]{GLSKM}}}]\label{lie1}
The cluster algebra $\cR(\cC_w)$
is a polynomial ring in the variables
$\delta_{M_1},\ldots,\delta_{M_r}$.
\end{Thm}

To obtain an alternative proof of 
Theorem~\ref{lie1}, one can proceed as follows:
\begin{itemize}
\setlength{\itemsep}{9pt}

\item[(i)]
Show that the cluster variables 
$\delta_{M_1},\ldots,\delta_{M_r}$ are algebraically independent.

\item[(ii)]
Show that for $1 \le k \le r$ the cluster variables 
$\delta_{V_k}$ and $\delta_{T_k}$ are polynomials in 
$\delta_{M_1},\ldots,\delta_{M_r}$.

\item[(iii)]
Apply Theorem~\ref{Thmfac}.

\end{itemize}
Part (i) can be done easily using induction and the mutation
sequence in \cite[Section~13]{GLSKM}.
Part (ii) is not at all straightforward. 

Let us give a concrete example illustrating Theorem~\ref{lie1}.
Let $\g$ be the Kac-Moody Lie algebra
associated to the generalized Cartan matrix
$$
C = \left(\bbm 2 & -2\\-2&2\ebm\right),
$$
and let $\bi = (2,1,2,1,2,1,2,1)$.
Then 
$\cA(\cC_w) = \cA(\bx,B_\bi)$, where
$r=n+2 = 8$,
$x_7$ and $x_8$ are the (non-invertible) coefficients, and 
$$
B_\bi =
\left(\bbm
0 & 2 &-1 & & & \\
 -2& 0  & 2& -1& & \\
1 &  -2 &0 &2 &-1 &  \\
 &   1& -2&0 &2 &-1 \\
 &   &1 & -2&0 &2 \\
 &   & & 1& -2& 0\\\hline
 &   & & & 1& -2\\
 & &  & & & 1 \\
\ebm\right).
$$
The principal part $B_\bi^\circ$ of $B_\bi$ 
is skew-symmetric, and the quiver
$\Gamma(B_\bi)$ looks as follows:
$$
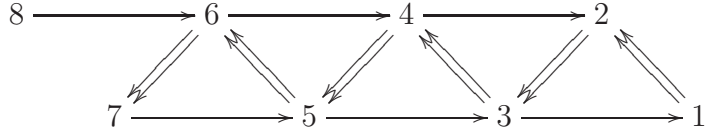
\xymatrix{
8\ar[rr]  && 6 \ar@<0.3ex>[dl]\ar@<-0.3ex>[dl]\ar[rr] 
&& 4 \ar@<0.3ex>[dl]\ar@<-0.3ex>[dl]\ar[rr] &&2\ar@<0.3ex>[dl]\ar@<-0.3ex>[dl] \\
& 7 \ar[rr]&& 5 \ar@<0.3ex>[ul]\ar@<-0.3ex>[ul]\ar[rr]&& 3\ar@<0.3ex>[ul]\ar@<-0.3ex>[ul]\ar[rr]& &1 \ar@<0.3ex>[ul]\ar@<-0.3ex>[ul]
}
$$
Define
\begin{align*}
(\bx[0],B[0]) &:= (\bx,B_\bi),\\
(\bx[1],B[1]) &:= \mu_5\mu_3\mu_1(\bx[0],B[0]), &
(\bx[2],B[2]) &:= \mu_6\mu_4\mu_2(\bx[1],B[1]),\\
(\bx[3],B[3]) &:= \mu_3\mu_1(\bx[2],B[2]),&
(\bx[4],B[4]) &:= \mu_4\mu_2(\bx[3],B[3]),\\
(\bx[5],B[5]) &:= \mu_1(\bx[4],B[4]),&
(\bx[6],B[6]) &:= \mu_2(\bx[5],B[5]),
\end{align*}
and for $0 \le k \le 6$ let
$(x_1[k],\ldots,x_8[k]) := \bx[k]$.

Under the isomorphism $\cA(\cC_w) \to \cR(\cC_w)$
the cluster $\bx[0]$ of $\cA(\cC_w) = \cA(\bx,B_\bi)$ corresponds to
the cluster $(\delta_{V_1},\ldots,\delta_{V_8})$ of $\cR(\cC_w)$, the cluster $\bx[6]$ 
corresponds to
$(\delta_{T_1},\ldots,\delta_{T_8})$, and we have
\begin{align*}
x_1[0] &\mapsto \delta_{M_1}, & 
x_2[0] &\mapsto \delta_{M_2}, & 
x_1[2] &\mapsto \delta_{M_3}, & 
x_2[2] &\mapsto \delta_{M_4}, &\\
x_1[4] &\mapsto \delta_{M_5}, &
x_2[4] &\mapsto \delta_{M_6}, &
x_1[6] &\mapsto \delta_{M_7}, &
x_2[6] &\mapsto \delta_{M_8}.
\end{align*}
By Theorem~\ref{lie1} we know that the cluster algebra
$\cA(\bx,B_\bi)$ is a polynomial ring in the variables
$x_1[0],x_2[0],x_1[2],x_2[2],x_1[4],
x_2[4],x_1[6],x_2[6]$.


\section{Applications}\label{Sec7}


\subsection{Prime elements in the dual semicanonical basis}
As in Section~\ref{lie} let $C \in M_{n,n}(\Z)$ be a symmetric generalized Cartan matrix,
and let
$\g = \n_- \oplus \h \oplus \n$
be the associated Lie algebra.

As before let $W$ be the Weyl group of $\g$.
To $C$ one can also associate a preprojective algebra $\LL$
over $\C$, see for example \cite{GLSKM,R}

Lusztig \cite{Lu} realized 
the universal enveloping algebra $U(\n)$ of $\n$
as an algebra of constructible functions 
on
the varieties $\LL_d$
of nilpotent $\LL$-modules with dimension vector $d \in \N^n$.
He also 
constructed the semicanonical basis $\cS$
of $U(\n)$.
The elements of $\cS$ are naturally parametrized by
the irreducible components of the varieties $\LL_d$.

An irreducible component $Z$ of $\LL_d$ is called \emph{indecomposable}
if it contains a Zariski dense subset of indecomposable 
$\LL$-modules, and $Z$ is \emph{rigid} if it
contains a rigid $\LL$-module $M$, i.e.
$M$ is a module with
$\Ext_\LL^1(M,M) = 0$. 

Let $\cS^*$ be the dual semicanonical basis of
the graded dual $U(\n)^*_{\rm gr}$ of $U(\n)$.
The elements $\rho_Z$ 
in $\cS^*$ are also parametrized by
irreducible components $Z$ of the varieties $\LL_d$.
We call $\rho_Z$ \emph{indecomposable} (resp. \emph{rigid})
if $Z$ is indecomposable (resp. rigid).
An element $b \in \cS^*$ is called \emph{primitive} if
it cannot be written as a product $b = b_1b_2$ with
$b_1,b_2 \in \cS^* \setminus \{1\}$.

\begin{Thm}[{{\cite[Theorem~1.1]{GLSSemi1}}}]
\label{Thmprim3}
If $\rho_Z$ is primitive, then $Z$ is indecomposable.
\end{Thm}

\begin{Thm}[{{\cite[Theorem~3.1]{GLSKM}}}]\label{Thmprim1}
For $w \in W$
all cluster monomials of the cluster algebra 
$\cR(\cC_w)$ belong to the dual semicanonical
basis $\cS^*$ of $U(\n)_{\rm gr}^*$. 
More precisely, we have
\begin{align*}
\{\text{cluster variables of 
}  \cR(\cC_w) \} &\subseteq \{ \rho_Z \in \cS^* \mid Z \text{ is indecomposable and rigid} \},\\
\{\text{cluster monomials of 
}  \cR(\cC_w) \} &\subseteq \{ \rho_Z \in \cS^* \mid Z \text{ is rigid} \}.
\end{align*}
\end{Thm}

Combining Theorems~\ref{Thmirr}, \ref{lie1}
and \ref{Thmprim1} 
we obtain a partial converse of Theorem~\ref{Thmprim3}.

\begin{Thm}\label{Thmprim2}
The cluster variables in $\cR(\cC_w)$ are prime,
and
they are primitive elements of $\cS^*$.
\end{Thm}

\begin{Conj}\label{primitive}
If $\rho_Z \in \cS^*$ is indecomposable and rigid, then $\rho_Z$ is prime in $U(\n)_{\rm gr}^*$.
\end{Conj}

\subsection{Monoidal categorifications of cluster algebras}
Let $\cC$ be an abelian tensor category with unit object
$I_\cC$.
We assume that $\cC$ is a Krull-Schmidt category, 
and that all objects in $\cC$ are of finite length.
Let $\cM(\cC) := \cK_0(\cC)$ be the 
Grothendieck ring of $\cC$.
The class of an object $M \in \cC$ is denoted by $[M]$.
The addition in $\cM(\cC)$ is given by
$[M] + [N] := [M \oplus N]$
and
the multiplication is defined by 
$
[M][N] := [M \otimes N].
$
We assume that $[M \otimes N] = [N \otimes M]$.
(In general this does not imply $M \otimes N \cong N \otimes M$.)
Thus $\cM(\cC)$ is a commutative ring.

Tensoring with $K$ over $\Z$ yields a $K$-algebra 
$\cM_K(\cC) := K \otimes_\Z \cK_0(\cC)$ with $K$-basis the classes
of simple objects in $\cC$.
Note that the unit object $I_\cC$ is simple.

A \emph{monoidal categorification} of a cluster algebra
$\cA(\bx,B)$ is an algebra isomorphism
$$
\Phi\colon \cA(\bx,B) \to \cM_K(\cC),
$$
where $\cC$ is a tensor category as above,
such that 
each cluster monomial $y = y_1^{a_1} \cdots y_m^{a_m}$
of $\cA(\bx,B)$ is mapped to a class $[S_y]$
of some simple object $S_y \in \cC$.
In particular, we have
$$
[S_y] = [S_{y_1}]^{a_1} \cdots [S_{y_m}]^{a_m} =
[S_{y_1}^{\otimes a_1} \otimes \cdots \otimes S_{y_m}^{\otimes a_m}].
$$
For an object $M \in \cC$ let $x_M$ be the element
in $\cA(\bx,B)$ with $\Phi(x_M) = [M]$.
 
The concept of a monoidal categorification of a cluster
algebra was introduced in \cite[Definition~2.1]{HL}.
But note that
our definition uses weaker conditions than in \cite{HL}.

An object $M \in \cC$ is called \emph{invertible} if
$[M]$ is invertible in $\cM_K(\cC)$.
An object $M \in \cC$ is \emph{primitive} if there are no
non-invertible objects $M_1$ and $M_2$ in $\cC$ with
$M \cong M_1 \otimes M_2$.

\begin{Prop}\label{monoid1}
Let $\Phi\colon \cA(\bx,B) \to \cM_K(\cC)$
be a monoidal categorification of a cluster 
algebra $\cA(\bx,B)$.
Then the following hold:
\begin{itemize}
\setlength{\itemsep}{9pt}

\item[(i)]
The invertible elements in $\cM_K(\cC)$ are
$$
\cM_K(\cC)^\times = 
\left\{ \la [I_\cC][S_{x_{n+1}}]^{a_{n+1}} \cdots [S_{x_p}]^{a_p} \mid \la \in K^\times, a_i \in \Z\right\}.
$$

\item[(ii)]
Let $M$ be an object in $\cC$ 
such that
the element $x_M$ is irreducible in $\cA(\bx,B)$.
Then $M$ is primitive.

\end{itemize}
\end{Prop}

\begin{proof}
Part (i) follows directly from Theorem~\ref{Thminvert}.
To prove (ii),
assume that $M$ is not primitive.
Thus  there are non-invertible objects 
$M_1$ and $M_2$ in $\cC$  with
$M \cong M_1 \otimes M_2$.
Thus in $\cM_K(\cC)$ we have $[M] = [M_1] [M_2]$.
Since $\Phi$ is an algebra isomorphism, we get
$x_M = x_{M_1}x_{M_2}$ with $x_{M_1}$ and $x_{M_2}$
non-invertible in $\cA(\bx,B)$.
Since $x_M$ is
irreducible, we have a contradiction.
\end{proof}

Combining Proposition~\ref{monoid1}
with Theorem~\ref{Thmirr} we get the following
result.

\begin{Cor}\label{monoid2}
Let $\Phi\colon \cA(\bx,B) \to \cM_K(\cC)$
be a monoidal categorification of a cluster 
algebra $\cA(\bx,B)$.
For each cluster variable $y$ of $\cA(\bx,B)$, the
simple object $S_y$ is primitive.
\end{Cor}

Examples of monoidal categorifications of
cluster algebras can be found in
\cite{HL,N1}, see also \cite{Le}.


\bigskip
{\parindent0cm \bf Acknowledgements.}\,
We thank Giovanni Cerulli Irelli,
Sergey Fomin, 
Daniel Labardini Fragoso, Philipp Lampe and Andrei Zelevinsky for
helpful discussions.
We thank Giovanni Cerulli Irelli for carefully reading 
several preliminary versions of this article.
The first author likes to thank the 
Max-Planck Institute for Mathematics in Bonn for a one year research stay in 2010/2011.
The third author
thanks the Sonderforschungsbereich/Transregio SFB 45 for financial 
support, and all three authors thank the Hausdorff Institute for Mathematics
in Bonn for support and hospitality.


\end{document}